\documentclass[12pt, reqno]{amsart}

\usepackage[utf8]{inputenc}
\usepackage[T1]{fontenc}
\usepackage{tikz,tikz-3dplot}
\usepackage{amsmath,amsfonts,amssymb,amsmath,latexsym,mathrsfs}
\usepackage{shuffle}
\usepackage[all]{xy}
\usepackage{stackengine}
 
\usepackage{enumerate}

\usepackage{hyperref}
 
\usepackage{tabu}
\usepackage{tikz-cd} 

\usepackage{comment}
\usepackage{url}
\usepackage{tikz}

\usepackage{xfrac}

\newtheorem{theorem}{Theorem}[section]

\newtheorem{proposition}[theorem]{Proposition}
\newtheorem{definition}[theorem]{Definition}

\newtheorem{corollary}[theorem]{Corollary}

\newtheorem{remark}[theorem]{Remark}

\newtheorem{lemma}[theorem]{Lemma}

\newcommand{\Z}{ {\mathbb Z} }
\newcommand{\C}{ {\mathbb C} }
\newcommand{\K}{ {\Bbbk} }

 \newcommand{\Q}{{\mathbf{Qu}}}
 \newcommand{\G}{{\mathbf{Grp}}}
\begin{document}

\title{Groups and quandles}

\author{Mohamad \textsc{Maassarani} }
\maketitle
\begin{abstract}
We are intereseted in quandles and their enveloping groups. Various results are proven. We show that a quandle $Q$ and its image in the enveloping group $G(Q)$ have isomorphic enveloping groups. The image quandle is injective. For $Q$ a finite quandle, we show that $G(Q)$ admits a faithfull representation $\rho : G(Q)\to GL_n(\Z)$ for some $n$; an irreducible representation of $G(Q)$ over $\C$ is finite dimensional an its degree divides the order of the group $Inn(Q)$ of inner automorphism of $Q$ and is bounded by $\sqrt{\vert Inn(Q) \vert }$. We determine the Malcev Lie algebra and the rational cohomology ring of $G(Q)$ for $Q$ finite. We prove that a finite injective quandle is a subquandle (for conjugacy) of a finite group. We also prove that the only finite subquandles (for conjugacy) of uniquely divisible groups are trivial quandles and that morphisms from quandles to nilpotent groups (for conjugacy) are constant on the indecomposable components. Implication of these results are considered.  
\end{abstract}
\section*{Introduction and main results}
A quandle is a set equipped with a binary operation satisfying some axioms. For instance, quandles are invariants of knots. Given a group $G$ one can construct a quandle $Conj(G)$ over the set $G$ by setting $x\triangleright y= xyx^{-1}$ for $x,y\in G$. A group morphism is a quandle morphism for the structure we just defined. Hence we have a functor $Conj$ from the category of groups $\G$ to the category of quandles $\Q$. For every quandle $Q$ there is a group $G(Q)$ called the enveloping group (or associated, adjpint, structure group) and quandle morphism $\varphi_Q:Q\to G(Q)$ which is universal with respect to quandle morphism from $Q$ into quandles of the form $Conj(G)$ for some group $G$. The assignement $Q\mapsto G(Q)$ is functorial and the corresponding functor $G:\Q \to \G$ is the left adjoint of $Conj$. These are known facts in the literature (\cite{DJ},\cite{Br},\cite{FR}) and will be reviewed. \\\\
In these notes we are interseted in the relation between quandles and groups and the enveloping groups of finite quandles. We have six short sections. We will summarize the content :\\\\
In the first section, we recall basic definitions as the definition of a quandle, the definition of a linear representation of a quandle. We alsp recall some basic facts on quandles and fix few notations.\\
In the second section we review the construction of $G(Q)$, the universal morphism $\varphi_Q : Q\to G(Q)$ for $Q$ a quandle and the adjuction betwenn $Conj$ and $G(-)$.\\\\
The map $\varphi_Q :Q\to G(Q)$ is not always injective. If it is the quandle is called injective. The third subsection is devoted to the image quandle $\varphi_Q(Q)$ which we denote by $Q_{conj}$. We show that $Q_{conj}$ is an injective quandle and that the enveloping groups $G(Q)$ and $G(Q_{conj})$ are naturaly isomorphic. The assignement $Q\to Q_{conj}$ is functorial and we have a natural equivalence betwen the functors $G(-)$ and $G\circ F(-)$, where $F$ is the functor corresponding to the assignement $Q\to Q_{conj}$.\\\\
We prove in section 4, that for $Q$ finite $G(Q)$ admits a faithfull representation $\rho : G(Q) \to GL_n(\Z)$, for some $n$. This proves that every finite injective quandle admits an injective representation $\rho : Q \to Conj(GL_n(\Z))$ for some $n$. The linearity of $G(Q)$ for $Q$ finite implies that $Q$ is a subquandle for conjugacy of some finite group. We also show that irreducible representations of a finite quandle $Q$ over $\C$ (equivalently irreducible representations of $G(Q)$) are finite dimensional and that the degree of such representations divides the order of the group $Inn(Q)$ of inner automorphisms of $Q$; moreover, the degree is bounded by $\sqrt{\vert Inn(Q) \vert }$.\\\\
In section 5, we show that a finite subquandle of a uniquely divisible group is trivial, (i.e. $x\triangleright y=y$) and derive some implications from it. We show that a quandle morphism to $Conj(G)$ for $G$ a nilpotent group is constant on indecomposable components of the source. We also prove that an injective quandle is nilpotent if and only if it is a subquandle for conjugacy of some nilpotent group. The group is finite if the quandle is finite, where we also describe the structure of $G(Q)$ for $Q$ a finite nilpotent quandle.\\\\
Finaly, we determine the Malcev Lie algebra and the rational cohomology ring of $G(Q)$ for $Q$ finite. This is done in the last section.

\section{Quandles, Quandle functor and linear representation of quandles}
\begin{definition}
\begin{itemize}
\item[1)] A quandle is set $Q$ equipped with a binary operation $\triangleright:Q\times Q \to Q$  such that :
\begin{itemize}
\item[a)] $x\triangleright x=x$ for $x\in Q$.
\item[b)] $x \triangleright( y\triangleright z)= (x \triangleright y) \triangleright (x \triangleright z)$, for $x,y,z \in Q$. 
\item[c)] For all $x,y \in Q$ there exist a unique $z\in Q$, such that $x \triangleright z=y$.
\end{itemize}
\item[2)] A subquandle of $(Q,\triangleright)$ is a subset $Q'$ such that $\triangleright$ restricts to quandle structure on $Q'$.
\item[3)] A quandle morphism is a map between two quandles $(Q,\triangleright_Q)$ and $(Q' ,\triangleright_{Q'})$ such that $f(z \triangleright_Q w)=f(z) \triangleright_{Q'} f(w)$, for $z,w \in Q$.
\end{itemize}
\end{definition}
The identity map of a quandle $Q$ is a quandle morphism and the composition of two quandle morphism is a quandle morphism. Hence quandles form a category $\Q$ whose objects are quandles and morphisms are quandle morphisms.
\begin{proposition}
Let $Q\to Q'$ be a quandle morphism. The image of $Q$ is a subquandle of $Q'$.
\end{proposition}
 For $G$ a group the binary operation $x\triangleright y= xyx^{-1}$ defines a quandle structure on $G$. The associated quandle will be denoted by $Conj(G)$ and will be called the \textit{conjugacy quandle of $G$} . A group morphism $f:G\to H$  is a quandle morphism with respect to the conjugacy quandle strucutre just defined. Hence, a group morphism is also a morphism $f: Conj(G)\to Conj(H)$ in the category of quandles. 
\begin{proposition}
Denote by $\G$ the category of groups. We have a functor $Conj: \G \to \Q$ defined by $G\mapsto Conj(G)$ and $Conj(f)=f$ for $G\in \G$ and $f$ a morphism in $\G$. We call $Conj$ the $\mathit{conjugacy\  quandle \ functor}$.
\end{proposition}
In \cite{rep}, the authors introduced linear representations and subrepresentations of quandles (more generally racks) :
\begin{definition}
\begin{itemize}
\item[1)] A linear representation of a quandle $Q$ over a vector space $V$ is a quandle morphism $\rho: Q \to Conj(\mathrm{GL}(V))$.
\item[2)] A subrepresentation of $\rho$ is a vector subspace $W\subset V$, such that $\rho(x)(W) \subset W$ for all $x\in Q$.
\item[3)] An irreducible representation is a representation $V$ that has no subrepresentations other than $0$ and $V$.
\end{itemize}
\end{definition}
\section{Adjoint functor to the conjugacy quandle functor}

For $X$ a set we denote by $FX$ the free group on $X$ and by $i_X :X\to FX$ the inclusion given by $i_X(x)=x$ for $x\in X$.
\begin{definition}
\begin{itemize}
\item[1)] For $Q$ a quandle, the conjugacy group $G(Q)$ of $Q$ is the group defined as the quotient of $FQ$ by the relations $xyx^{-1}(x\triangleright y)^{-1}=1$, for $x,y\in Q$.
\item[2)] For $Q$ a quandle $\varphi_Q :Q\to G(Q)$ is the composite of $i_Q$ by the quotient map $FQ\to G_Q$.
\end{itemize}
\end{definition}
The group $G(Q)$ is sometimes called the enveloping group of $Q$, the associated group, adjoint group or structure group. For $Q$ the trivial quandle ($x\triangleright y=y$) the group $G(Q)$ is isomorphic to $\Z Q$ the free abelian group on $Q$. 
\begin{proposition}
Let $Q$ be a quandle.
\begin{itemize}
\item[1)] The group $G(Q)$ is infinite.
\item[2)] The elements of $\varphi_Q(Q)$ and products of these elements are of infinite order.
\item[3)] If $Q$ is finite (as a set) then $G(Q)$ is a finitely presented group.
\item[4)] The map $\varphi_Q:Q\to G(Q)$, is a quandle morphism if $G(Q)$ is regarded as the conjugacy quandle $Conj(G(Q))$.
\end{itemize}
\end{proposition}
\begin{proof}
We prove $1)$. Let $L:FQ \to \Z$ be the morphism defined by $x\mapsto 1$ for $x\in Q$. For $x,y \in Q$, the elements $xyx^{-1}(x\triangleright y)^{-1}$ lie in the kernel of $L$. Therefore, $L$ induces a morphism $\bar{L} : G(Q)\to \Z$ mapping $\varphi_Q(x)$ to $1$ for $x\in Q$. $\bar{L}$ is hence surjective and $G(Q)$ is hence infinite. $2)$ follows from the fact that $\bar{L}(\varphi_Q(x))=1$, for $x\in Q$. $3)$ follows immediatly from the definition of $G(Q)$. $4)$ follows from the fact that, for $x,y\in Q$, $\varphi_Q(x\triangleright y)=xyx^{-1}$ in $G(Q)$ because of the relation $xyx^{-1}(x\triangleright y)^{-1}=1$.
\end{proof}

\begin{proposition}\label{univ}
Let $Q$ be a quandle, $G$ be a group and $f:Q\to Conj(G)$ be a quandle morphism. There exist a unique group morphism $G_f:G(Q)\to G$ such that the following diagram commutes :
$$\begin{tikzcd}
G(Q) \arrow{r}{G_f} &G  \\ 
Q \arrow{u}{\varphi_Q} \arrow{ru}{f} & 
\end{tikzcd} $$
\end{proposition}
\begin{proof}
Uniqueness is due to the fact that $\varphi_Q(X)$ generates $G(Q)$. We prove the existence. Let $F_f:FQ\to G$ be the group morphism mapping $x$ to $f(x)$ for $x\in Q$. For $x,y \in Q$, 
$$F_f(xyx^{-1}(x\triangleright y)^{-1})=f(x)f(y)f(x)^{-1}f(x\triangleright y)^{-1}=f(x)f(y)f(x)^{-1}(f(x)f(y)f(x)^{-1})^{-1}=1.$$
Hence, the elements $xyx^{-1}(x\triangleright y)^{-1}$ lie in the kernel of $F_f$ and therefore $F_f$ induces a group morphism $G_f :G(Q)\to G$ such that $G_f(\varphi_Q(x))=F_f(x)=f(x)$, for $x\in Q$. We have proved the proposition.
\end{proof}
\begin{remark}
With the last proposition one sees that linear representations of a quandle $Q$ correspond to linear representations of the group $G(Q)$. As one can check a subrepresentation for $Q$ correspond to a subrepresentation of $G(Q)$. Hence, the study of linear representations of a quandle $Q$ is related to linear representations of the infinite group $G(Q)$.  Since for $Q$ finite $G(Q)$ is finitely presented one has representation varieties associated to $Q$ finite, the representation varieties of $G(Q)$ $($\cite{Mag}$)$
\end{remark}
\begin{corollary}\label{fun}
Let $f:Q\to Q'$ be a quandle morphism
\begin{itemize}
\item[1)] There exist a unique group morphism $G(f) : G(Q) \to G(Q')$ such that the following diagram commutes : 
$$\begin{tikzcd}
G(Q) \arrow{r}{G(f)} &G(Q')  \\ 
Q\arrow{u}{\varphi_Q} \arrow{r}{f} & Q'\arrow{u}{\varphi_{Q'}} 
\end{tikzcd} $$
\item[2)] The morphism $G(f)$ is given by $G(f)(\varphi_Q(x))=\varphi_{Q'}(f(x))$.
\end{itemize}
\end{corollary}
\begin{proof}
This is an application of the previous proposition since $\varphi_{Q'}$ is a quandle morphism. We have $G(f)=G_{\varphi_Y \circ f}$. The fact the equation in $2)$ defines the morphism is due to the fact that $\varphi_Q(Q)$ generates $G(Q)$.
\end{proof}

\begin{corollary}\label{fund}
We have a faithfull functor $G: \Q \to \G$ given by $Q\mapsto G(Q)$ and $f\mapsto G(f)$ for $Q\in \Q$ and $f$ a morphism in $\Q$.
\end{corollary}
\begin{proof}
The uniqueness of $G(f)$ in the previous corollary implies that $G(f_1\circ f_2)=G(f_1)\circ G(f_2)$ for $f_1$ $f_2$ composable morphisms in $\Q$. Point $2)$ of the previous corollary implies that $G(\mathrm{id}_Q)=\mathrm{id}_{G(Q)}$ since $\varphi_Q(Q)$ generates $G(Q)$.
\end{proof}
For $Q \in \Q$ and $G \in \G$ define $\theta_{Q,G}: Hom_\G(G(Q),G) \to Hom_\Q(Q,Conj(G))$ by $\theta_{Q,G}(f)=f\circ \varphi_Q$.
\begin{lemma}
\begin{itemize}
\item[1)] For all $Q\in \Q$ and $G\in\G$, $\theta_{Q,G}$ is a bijection.
\item[2)] $\theta_{Q,G}$ is natural in $Q$ and $G$.
\end{itemize}
\end{lemma}
\begin{proof}
 The morphism $\theta_{Q,G}$ is surjective by the existence part of proposition \ref{univ} and is injective by the uniqueness part of the same proposition. This proves $1)$. The naturality of $\theta_{Q,G}$ in $Q$ is equivalent to the statement :
$$ h\circ \varphi_{Q'} \circ f=h\circ G(f) \circ \varphi_Q,$$
for $h\in Hom_\G(G(Q'),G)$ and $f\in Hom_\Q(Q,Q')$. But this is true by corollary \ref{fun}.
 The naturality of $\theta_{Q,G}$ in $G$ is equivalent to the statement :
$$ h \circ f \circ \varphi_Q=Conj(h)\circ f \circ \varphi_Q,$$
for $f\in Hom_\G(G(Q),G)$ and $h\in Hom_\Q(G,G')$. But this is true since $Conj(h)=h$.
\end{proof}
\begin{theorem}
The triple $\langle G, Conj, \theta \rangle$ is an adjuction from $\Q$ to $\G$.
\end{theorem}
We note that given a groupe $G$ the binary operation $x\triangleright y=xy^{-1}x$ defines a quandle structure on $G$. Usually the associated quandle is denoted by $Core(G)$. A group morphism is a quandle morphism with respect to this operation and hence we have a functor $Core : \G \to \Q$. We can associate to a quandle $(Q,\triangleright_Q)$ a group $G(Q)^{core}$ defined as the quotient of the free group $FQ$ by the relations $xy^{-1}x(i_Q(x\triangleright_Q y))^{-1}=1$ for $x,y\in Q$. One can prove for $G(Q)^{core}$ and $Core$ results similar to those obtained for $G(Q)$ and $Conj$.\\

\section{Injective quandles and an intermidiate functor}\label{S3}
One can wonder wether $\varphi_Q :Q\to G(Q)$ is always injective. The answer is known (\cite{DJ}) to be no. We will give an example from \cite{DJ} and a general argument.
\begin{proposition}
Let $Q$ be a quandle having two elements $x,y$ such that $x\triangleright y=y$ and $y\triangleright x\neq x$. The morphism $\varphi_Q :Q\to G(Q)$ is not injective, in fact $\varphi_Q(x)=\varphi_Q(y\triangleright x)$.
\end{proposition}
\begin{proof}
We have in $G(Q)$ that : 
$$yxy^{-1}=y\triangleright x \quad \text{and} \quad yxy^{-1}=y(xy^{-1}x^{-1})x=y(x\triangleright y)^{-1}x =x,$$
Hence, in $G(Q)$, $y\triangleright x=x$. This proves the proposition.
\end{proof}
The same holdes for $G(Q)^{core}$ if we write $$yx^{-1}y=y\triangleright x\quad \text{and} \quad yx^{-1}y=y(x^{-1}yx^{-1})x=y(x\triangleright y)^{-1}x =x.$$
As one can check the hypothesis of the previous propostition don't hold for a quandle $Conj(G)$ for any group $G$. In such quandles if $x\triangleright y=y$ then $y\triangleright x =x$.\\
We now give an example of a quandle satisfying the hypothesis in the proposition.
Let $Q_3$ be the set on three elements $\{x,y,z\}$ and define $\triangleright :Q_3 \times Q_3\to Q_3$, by 
$$ y  \triangleright  x=z , \ y\triangleright z =x , a\triangleright  b=b, \  \text{for $a\in\{x,z\}$ and $b\in Q_3$}.$$

As one can check $Q_3$ equipped with the above binary operation is a quandle. Moreover, $Q_3$ satisfies the hypothesis of the previous proposition ($x\triangleright y=y$ and $y\triangleright x=z\neq x$). Hence, $\varphi_{Q_3}: Q_3 \to G_{Q_3}$ is not injective and $\varphi_{Q_3}(x)=\varphi_{Q_3}(z)$. It follows that every quandle morphism from $Q_3$ to a group is not injective. In particular, $Q_3$ has no injective linear representation. Moreover, if we denote by $Q_2$ the subquandle $\{x,z\}\subset Q_3$. We have a  subquandle satisfying $Q_3\triangleright Q_2 \subset Q_2$, such that the morphism $G_{Q_2}\to G_{Q_3}$ induced by the inclusion $Q_2\to Q_3$ is not injective. We also have that an injective representation of $Q_2$ don't lift to a representation of $Q_3$.  

\begin{definition}
A quandle $Q$ is called injective if $\varphi_Q :Q \to G(Q)$ is injective. 
\end{definition}
\begin{proposition}
For $Q$ a quandle, the morphism $\varphi_Q:Q\to G(Q)$ is injective if and only if $Q$ is a subquandle of $Conj(G)$ for some group $G$.
\end{proposition}
\begin{proof}
If $\varphi_Q$ is injective then $Q$ is "isomorphic" to the subquandle $\varphi_Q(Q)$ of $Conj(G(Q))$ and this proves the only if part. If $Q$ is a subquandle of $Conj(G)$ for some $G$. Then the morphism $f:G(Q)\to G$ factoring the inclusion $Q\to Conj(G)$ satisfies $f(\varphi_Q(x))=x$ hence $\varphi_Q$ is injective. We have proved the proposition.
\end{proof}
We recover the obvious facts, that quandles not comming out of groups do not admit injective quandle morphisms into groups endowed with conjugacy quandle structre. In paticular, non injective quandles admet no injective representations.
\begin{definition}
$\Q^{inj}$ is the full subcategory of $\Q$ whose objects are subcandles of $Conj(G)$ for $G$ running over all objects of $\G$ or equivalentely the full subcategory of injective quandles.
\end{definition}
 
\begin{definition}
For $Q$ a quandle, we definie the quandle $Q_{conj}$ as the subquandle $\varphi_Q(Q)$ of $Conj(G(Q))$, and we define the quandle morphism $\theta_Q :Q\to Q_{conj}$ by $\theta_Q(x)=\varphi_Q(x)$.
\end{definition}
\begin{proposition}
For $Q$ a quandle $Q_{conj}$ is an object of $\Q^{inj}$. In particular $\varphi_{Q_{conj}}:Q_{conj}\to G(Q_{conj})$ is injective.
\end{proposition}
\begin{proof}
Follows from the previous proposition since $Q_{conj}$ is a subquandle of $Conj(G(Q))$.
\end{proof}
Now Let $f:Q\to Q'$ be a quandle morphism, by corollary \ref{fun} there is a unique group morphism $G(f) : G(Q) \to G(Q')$  such that   the following diagram commutes : 
$$\begin{tikzcd}
G(Q) \arrow{r}{G(f)} &G(Q')  \\ 
Q\arrow{u}{\varphi_Q} \arrow{r}{f} & Q'\arrow{u}{\varphi_{Q'}} 
\end{tikzcd} $$
In particular $G(f)$ restricts to a morphism $f_{conj} :Q_{conj} \to Q'_{conj}$, such that the following diagram commutes :
\begin{equation}\label{diag}\begin{tikzcd}
Q_{conj} \arrow{r}{f_{conj}} &Q'_{conj}  \\ 
Q\arrow{u}{\theta_Q} \arrow{r}{f} & Q'\arrow{u}{\theta_{Q'}} 
\end{tikzcd} 
\end{equation}

It follows that :
\begin{proposition}
We have a functor $F:\Q \to \Q^{inj}$ given by $F(Q)=Q_{conj}$ and $F(f)=f_{conj}$ for $Q$ a quandle and $f$ a quandle morphism.
\end{proposition}
The following proposition establishes some properties of $Q_{conj}$.
\begin{proposition}\label{Gcon}
Let $Q$ be a quandle. 
\begin{itemize}
\item[1)] The map $\theta_Q : Q \to Q_{conj}$, is a quandle morphism universal for quandle morphisms $Q\to Conj(G)$ where $G$ is a group.
\item[2)] The inclusion $Q_{conj}\to G(Q)$ is universal with respect to quandle morphisms $Q_{conj}\to Conj(G)$ for $G$ a group.
\item[3)] The morphism $G(\theta_Q) : G(Q) \to G(Q_{conj})$ is an isomorphism and $\varphi_{Q_{conj}} :Q_{conj} \to G({Q_{conj}})$ is injective.  
\end{itemize}
\end{proposition}
\begin{proof}
We prove $1)$. Let $Q\to Conj(G)$ be a quandle morphism there exist a unique group morphism $G_f : G(Q) \to G$ such that $G_f\circ \varphi_Q= f$. The morphism $G_f$ restricts to a quandle morphism  $g :Q_{conj} \to Conj(G)$ satisfying $g\circ \theta_Q=f$. This proves the existence, the uniqueness is due to the fact that $\theta_Q$ is surjective. Let us show$2)$. Let $f:Q_{conj} \to Conj(G)$ be a quandle morphism. A morphism extending $f$ with respect to the inclusion $Q_{conj}\to G(Q)$ is a morphism extending $f\circ \theta_Q : Q\to Conj(G)$. Hence, such morphism exists and is unique by the universal properpty of $\varphi_Q$. Let us consider $3)$. We have already seen that $Q_{conj}$ is injective. $G(\theta_Q)$ is a morphism making the following diagram commutative :
$$\begin{tikzcd}
G(Q) \arrow{r}{G(\theta_Q)} &G(Q_{conj})  \\ 
Q\arrow{u}{\varphi_Q} \arrow{r}{\theta_Q} & Q_{conj}\arrow[swap]{u}{\varphi_{Q_{conj}}} 
\end{tikzcd} $$
In particular the $G(\theta_Q)$ is the group morphism extending $\varphi_{Q_{conj}}$ with respect to the inclusion $Q_{conj}\to G(Q)$ : 
$$\begin{tikzcd}
G(Q) \arrow{r}{G(\theta_Q)} &G(Q_{conj})  \\ 
Q_{conj}\arrow{u} \arrow[swap]{ru}{\varphi_{Q_{conj}}} &  
\end{tikzcd} $$
Hence, this is an isomorphism since $Q_{conj}\to G(Q)$ and $\varphi_{Q_{conj}}:Q_{conj}\to G(Q_{conj})$ are both universal maps for quandle morphism into groups.
\end{proof}
It follows from the example of the quandle $Q_3$ and the last point in the proposition that :
\begin{corollary}
The functor $G(-)$ is not essentially injective on objects. 
\end{corollary}
The corollary is known in the litterature (for instance \cite{DJ})
\begin{remark}
It follows from $1)$ of the previous proposition that for instace linear representations (resp. subrepresentations) of $Q$ correspond to linear representation (resp. subrepresentations) of $Q_{conj}$. 
\end{remark}
\begin{proposition}
For $Q$ a quandle $G(Q)$ is abelian if and only if $Q_{conj}$ is trivial $($i.e. $x\triangleright y= y$ for all $x,y\in Q$$)$.
\end{proposition}
\begin{proof}
From the last proposition (point $3)$), $G(Q)\simeq G(Q_{conj})$. One can check from the definition of the group $G(Q_{conj})$ that the group is abelian if $Q_{conj}$ is trivial. If the group $G(Q_{conj})$ is abelian then every subquandle (for conjugacy) of the group is trivial, in particular the subquandle $Q_{conj}$ ($Q_{conj}$ is injective) is trivial.
\end{proof}
\begin{corollary}\label{ab1}
If $G(Q)$ is abelian then $G(Q)$ is isomorphic to $\Z Q_{conj}$ the free abelian group on the set $Q_{conj}$.
\end{corollary}
In \cite{ab} the authors describe the finite quandles with abelian enveloping group.
\begin{proposition}
The collection of isomorphisms $G(\theta_Q):G(Q) \to G(Q_{conj})$ for $Q\in Q$ defines a natural trasformation from the functor $G(-)$ to the functor $G\circ F(-)$. 
\end{proposition}
\begin{proof}
Let $f:Q\to Q'$ be a quandle morphism. As we have seen in the proof of proposition \ref{Gcon} the following diagram is commutative :
 $$\begin{tikzcd}
G(Q) \arrow{r}{G(\theta_Q)} &G(Q_{conj})  \\ 
Q_{conj}\arrow{u} \arrow[swap]{ru}{\varphi_{Q_{conj}}} &  
\end{tikzcd} $$
where $Q_{conj}\to G(Q)$ is the inclusion. Combining this with diagram (\ref{diag}) and the diagram of corollary \ref{fun} for $Q_{conj}$ and $Q'_{conj}$, we get the followin commutative diagram :
$$\begin{tikzcd}
G(Q) \arrow{r}{G(\theta_Q)} & G(Q_{conj}) \arrow{r}{G\circ F (f)} &G(Q'_{conj}) & \arrow[swap]{l}{G(\theta_{Q'})} G(Q')  \\ 
&\arrow{ul}Q_{conj}\arrow[swap]{u}{\varphi_{Q_{conj}}} \arrow{r}{F(f)}  &  Q'_{conj}\arrow{u}{\varphi_{Q'_{conj}}}\arrow[swap]{ur}& \\
&\arrow{uul}{\varphi_{Q}}Q\arrow{u}{\theta_Q}\arrow{r}{f}    &  Q'\arrow{u}{\theta_{Q'}}\arrow[swap]{uur}{\varphi_{Q'}}&
\end{tikzcd} $$
Therefore, $G(\theta_Q)( G\circ F (f) )G(\theta_{Q'})^{-1}=G(f)$ by uniqueness of $G(f)$ from corollary \ref{fun}. We have proved the proposition.
\end{proof}

\section{Enveloping group of finite quandles and representations}
For $Q$ a quandle it follows from the definition of a quandle we took that for $x\in Q$  the map $L_x: Q\to Q$ given by $L_x(y)=x\triangleright y$ is a bijective quandle morphism. The group $Inn(Q)$ is the subgroup of the group of bijections of $Q$ generated by the maps $L_x$ for $x\in Q$. 

\begin{lemma}[\cite{dar}]\label{lem1}
Let $Q$ be an injective quandle. $G(Q)$ is an extension of $Inn(Q)$ by the center $\mathcal{Z}(G(Q))$ of $G(Q)$.
\end{lemma}
\begin{proof}
Let $x$ be an element of $Q$ seen as an element of $G(Q)$. The conjugacy by $x$ in $G(Q)$ stabilises $Q \subset G_Q$. Indeed, $xyx^{-1}=L_x(y)$ for every $y\in Q$. It follows, since $Q$ generates $G(Q)$ that we have a surjective morphism from $h: G(Q)\to Inn(Q)$ associating to $a\in G(Q)$ the restriction of the conjugacy by $a$ to $Q\subset G(Q)$. Since $Q$ generate $G(Q)$ the kernel of $h$ is exactly the center of $G(Q)$.  
\end{proof}
\begin{corollary}\label{cen}
For any quandle $Q$, $G(Q)\simeq G(Q_{conj})$ is an extension of $Inn(Q_{Conj})$ by the center $\mathcal{Z}(G(Q))$ of $G(Q)$.
\end{corollary}
\begin{lemma}\label{lem2}
Every finitely generated abelian group admits a faithfull representation into $GL_n(\Z)$ for some $n$.
\end{lemma}
\begin{proof}
For $m>0$ denote by $\rho_m: \Z/m\Z \to GL_m(\Z)$ the representation given by :
$$ \rho_m(1)=   \begin{bmatrix}
  0& 0 	& \dots &0&1\\
1&\ddots &\ddots &\ddots&0  \\
0&\ddots & \ddots & \ddots& \vdots \\
\vdots& \ddots & \ddots & \ddots &\vdots\\
0& &0 &1 &  0\end{bmatrix}
$$
It is a faithfull representation of $\Z/m\Z$. For $m=0$, denote by $\rho_0 : \Z \to GL_2(\Z)$ the representation given by :
$$\rho_0(1)= \begin{bmatrix}
  1& 1\\
0&1 \\
\end{bmatrix} $$
$\rho_0$ is a faithfull representation. A finitely generated abelian group is isomorphic to :
 $$G=\Z/{m_1}\Z\times \Z/m_2\Z \times \cdots \times \Z/m_k \Z.$$
for some $m_i\geq 0$. Let $\theta_i$ be the representation of $G$ given by $$\theta_i(a_1,\dots , a_k)=\rho_{m_i}(a_i),$$ for $(a_1,\dots,a_k)\in G$. One checks that the representation : $$\overset{k}{\underset{i=1}{\bigoplus}} \theta_i,$$
is a faithfull representation of $G$ into some $GL_n(\Z)$ for some $n$. We have proved the lemma.
\end{proof}
\begin{theorem}\label{rep}
For $Q$ be a finite quandle, there exist a faithfull group representation $\rho: G(Q)\to GL_n(\Z)$ for some $n$. 
\end{theorem}
\begin{proof}
Since we have proved in the previous section that $G(Q)\simeq G(Q_{conj})$ we can assume that $Q$ is injective. Since $Q$ is finite $G(Q)$ is finitely generated and $Inn(Q)$ is finite. Hence by the lemma \ref{lem1} the center $\mathcal{Z}(G(Q))$ of $G(Q)$ is a finite index subgroup of a finitely generated group and therefore $\mathcal{Z}(G(Q))$ is a finitely generated subgroup. Therefore by lemma \ref{lem2} $\mathcal{Z}(G(Q))$ admits a faithfull representation $\rho' : \mathcal{Z}(G(Q)) \to GL_m(\Z)$ for some $m$. Since  $\mathcal{Z}(G(Q))$ has finite index and $\rho'$ is faithfull, the induced representation $Ind_ {\mathcal{Z}(G(Q))}^{G_Q} \rho'$ is with values in $GL_n(\Z)$ for some $n$ and is faithfull. We have proved the theorem.
\end{proof}
We have already seen that non injective quandles admit no injective linear representations.
\begin{corollary}
A finite injective quandle $Q$ admits an injective representation $\rho : Q\to GL_n(\Z)$ for some $n$. 
\end{corollary}
\begin{proposition}\label{fin}
A finite injective quandle $Q$ is a subquandle of some $Conj(G)$ for $G$ finite.
\end{proposition}
\begin{proof}
It follows from the previous theorem that $G(Q)$ is residually finite (\cite{Mal}). Hence, for every couple $(x,y)\in Q^2$ there is a finite quotient $G(x,y)$ of $G(Q)$ in wich the class of $xy^{-1}$ is not equal to $1$. In particular, the classes of $x$ and $y$ are distinct in $G(x,y)$. The composed group morphism $G(Q) \to \prod_{(x,y)\in Q^2} G(x,y)$ restricted to $Q$ is an injective quandle morphism and $ \prod_{(x,y)\in Q^2} G(x,y)$  is a finite group.
\end{proof}

\begin{lemma}
Let $G$ be a countable group and $V$ an irreducible represenation of $G$ over $\C$. The center of $G$ acts on $V$ by multiplication by scalars.
\end{lemma}
\begin{proof}
$V$ is irreducible. Hence, any vector is cyclic and since $G$ is countable $V$ has countable dimension over $\C$. Therefore, the lemma follows from Dixmier's lemma.
\end{proof}
\begin{proposition}
Let $G$ be a countable group with center $\mathcal{Z}(G)$ of finite index. An irreducible representation of $G$ over $\C$ is finite dimensional and the degree of such represenation is bounded by $\sqrt{\vert G/\mathcal{Z}(G) \vert}$.
\end{proposition}
\begin{proof}
Let $\theta : G \to GL(V)$ be an irreducible representation of a group as in the proposition. Let $s : G/\mathcal{Z}(G) \to  G$ be a section and $v$ be a cyclic vector of $V$. By the previous lemma the center acts by multiplication by scalars and hence the span of the vectors $gv$ for $g\in G$ is equal to the span of the vectors $s(h)v$ for $h\in G/\mathcal{Z}(G)$. Since the center has finite index the representation is finite dimensional. Let $\rho : \C[G] \to End(V)$ be the algebra morphism corresponding to the irreducible representation $\theta$ of $G$. Since $V$ is finite dimensional by the density theorem $\rho$ is surjective. On the other hand, since the center acts by scalars, the image of $\rho$ is linearly generated by $\rho(s(h))$ for $h\in G/\mathcal{Z}(G)$. Hence, $dim(V)^2 \leq \vert G/ \mathcal{Z}(G)\vert $. We have proved the proposition. 
\end{proof}
\begin{proposition}
For $Q$ a finite quandle, an irreducible representations of $Q$ or equivalently of $G(Q)$ is finite dimensional; the degree of such representation is bounded by $\sqrt{\vert Inn(Q) \vert }$. 
\end{proposition}
\begin{proof}
follows from the last theorem, the fact that $G(Q)$ for finite quandles is finitely generated, corollary \ref{cen} and the fact (that we will see later) that we have a surjective morphism from $Inn(Q)\to Inn(Q_{conj})$.
\end{proof}

\begin{proposition}
Let $Q$ be a finite quandle and $V$ a finite dimensional irreducible representation of $G(Q)$ (equivalently of $Q$) over $\C$. The degree of $V$ divides the order of $Inn(Q_{conj})$ and hence divides the order of $Inn(Q)$.
\end{proposition}
\begin{proof}
We will see that we have a surjective morphism from $Inn(Q)$ into $Inn(Q_{conj})$. Hence, the order of $Inn(Q_{conj})$ divides the order of $Inn(Q)$. We only need to prove that the degree of $V$ divides the order of $Inn(Q_{conj})$. Since $V$ is a finite dimensional irreducible representation the center of $G(Q)$ acts by scalars in $V$ and hence our representation induces an irreducible projective representation $\rho : Inn(Q_{conj}) \to PGL(V)$ (see corollary \ref{cen}). The proposition follows from the fact that over $\C$ the degree of a finite dimensional projective irreducible representation divides the order of the group. Indeed, If $\theta : G \to PGL(W)$ is an irreducible finite dimensional representation then the group $\tilde{G}$ in $SL(W)$ (special linear group) preimage of $\theta(G) \subset PGL(V)$ is a finite central extension of $\theta(G)$. This is true since the natural morphism $SL(W)\to PGL(W)$ is surjective with finite kernel lying in the center of $SL(W)$. Now the natural inclusion $\tilde{G}\to GL(W)$ is an irreducible finite dimensional representation of a finite group over 
$\C$. Hence, the dimension of $W$ divides the index of the center of $\tilde{G}$. But the quotient of $\tilde{G}$ by its center is a quotient of $\theta(G)$ and hence a quotient of $G$. The claim follows.
\end{proof}
\begin{remark}
The map $Q\to Inn(Q),x\mapsto L_x$ is a quandle morphism and the image generates $Inn(Q)$ hence any irreducible represenation of $Inn(Q)$ gives an irreducible represenation of $Q$. Moreover, an irreducible representation over $\C$ of $Q$ finite induces an irreducible projective represenation of $Inn(Q_{conj})$ and hence of $Inn(Q)$.
\end{remark}
\begin{remark}
A finite quandle have infinitely many irreducible representation. For instance, as we will see the abelianisation of $G(Q)$ is a free abelian group and hence $G(Q)$ and $Q$ have infinitely many non equivalent $1$-dimensional representation. 
\end{remark}
\begin{remark}
The group $G(Q)$ is of type I, since it is countable and contains an abelian subgroup of finite index, its center.
\end{remark}
\section{Quandles, uniquely divisible groups and nilpotent groups}

\begin{proposition}
Let $Q$ be a finite quandle and $f:Q\to Conj(G)$ a quandle morphism where $G$ is a uniquely divisible group. 
\begin{itemize}
\item[1)] The image of $f$ is a trivial quandle.
\item[2)] The morphism $f$ factors throught $Q\overset{\varphi_Q}{\to} G(Q)\to G(Q)^{ab}$, where $G(Q)\to G(Q)^{ab}$ is the abelianisation map.
\end{itemize}
\end{proposition}
\begin{proof}
For $x$ and $y$ in $Q$ , denote by $x^n\triangleright y$ the $n$ times iteration $(x\triangleright(x\triangleright (\cdots (x\triangleright(x \triangleright y))\cdots)))$. Since $Q$ is finite $x^m\triangleright y=y$ for some $m$ ($L_x$ is of finite order). Hence : 
$$f(x^m\triangleright y)=f(x)^mf(y)f(x)^{-m}=f(y) \quad \text{and}\quad f(x)^mf(y)^mf(x)^{-m}=f(y)^m.$$
Since $G$ is uniquely divisible $f(x)f(y)f(x)^{-1}=f(y)$. This proves the proposition.
\end{proof}
The group $G(Q)^{ab}$ can be determined by the orbits of the quandle (see next section). 
\begin{corollary}
The only finite subquandles (for conjugacy) of a uniquely divisible group are trivial. 
\end{corollary}
\begin{corollary}
The only finite subquandle of the group $U_n(\K)$ of unipotent matrices of dimension $n\times n$ over a characteristic $0$ field are trivial.
\end{corollary}
\begin{corollary}
For $\K$ a field of characteristic $0$ and $Q$ a finite quandle, $G(Q)$ can be embeded into $U_n(\K)$ for some $n$ if and only if $Q_{conj}$ is trivial (equivalently $G(Q)$ is abelian, $G(Q)\simeq \Z Q_{conj})$. 
\end{corollary}
\begin{proof}
The equivalence after equivalently was established in section \ref{S3} corollary \ref{ab1}. If $G(Q)$ can be embeded $Q_{conj}$ can be embeded hence $Q_{conj}$ is trivial. In the other direction if $Q_{conj}$ is trivial $G(Q)\simeq G(Q_{conj}) \simeq \Z Q_{conj}$ and such group can be imbeded in $U_n(\K)$ for some $n$.
\end{proof}
For $Q$ a quandle the orbit of $x\in Q$ is the orbit of $x$ with respect to $Inn(Q)$.
\begin{definition}
A quandle $Q$ is decomposable if it is the disjoint union of two (nonempty) subquandles otherwise $Q$ is indecomposable.
\end{definition}
The points of the following proposition are known. They can be found either in the proof of proposition 1.17, either correspond to lemma 1.15  of \cite{deco} :
\begin{proposition}\label{ind}
Let $Q$ be a quandle and $x$ an element of $Q$.
\begin{itemize}
\item[1)] $Q$ is indecomposible if and only if the orbit of $x$ is equal to $Q$.
\item[2)] An two indecomposable subquandles of $Q$ containing $x$ are contained in a third indecomposable subquandle.
\item[3)] There is a maximal (for inclusion) indecomposable subquandle $Q_x$ of $Q$ containing $x$.
\item[4)] The subquandles $Q_y$ for $y\in Q$ form a partition of $Q$.
\end{itemize}
\end{proposition}

Here is a well known fact :
\begin{proposition}\label{abel}
We have a group morphism $G(Q) \to \Z (Q/Inn(Q))$ (the free abelian group on $Q/Inn(Q)$) mapping the image of $x\in Q$ in $G(Q)$ to its class in $Q/ Inn(Q)$. This morphism induces an isomorphism $G(Q)^{ab}\to \Z (Q/Inn(Q))$.
\end{proposition}
\begin{proof}
This can be maneged using generators and relations.
\end{proof}
\begin{corollary}\label{clas}
If $Q$ is indecomposable then the classes of $x,y\in Q$ in $G(Q)^{ab}$ are equal.
\end{corollary}
We will denote by $grG(Q)=\oplus_{i\geq1} gr_i G(Q)$ the associated graded of $G(Q)$ with respect to the lower central series $(\Gamma_n G(Q))_n$. We recall that the associated graded of a group  with respect to the lower central series is a graded Lie ring generated by its degree $1$ part ($gr_1$) and that the Lie bracket is induced by the commutator.
\begin{lemma}
For $Q$ indecomposable $grG(Q)\simeq gr_1 G(Q)\simeq \Z$ $($$gr_i G(Q)=0$ for $i>1$$)$.
\end{lemma}
\begin{proof}
By $1)$ of proposition \ref{ind} $Q/Inn(Q)$ is reduced to one point. Therefore, by the previous proposition $G(Q)^{ab}=gr_1G(Q)$ is isomorphic to $\Z$. Since $grG(Q)$ is generated as a Lie ring by its degree $1$ part and the degree $1$ part is reduced to $\Z$, $gr_iG(Q)=0$ for $i>1$ (the Lie bracket is antisymetric). 
\end{proof}
\begin{proposition}
Let $Q$ be a quandle and $f:Q\to Conj(G)$ be a quandle morphism where $G$ is nilpotent. The morphism $f$ is constant on the maximal indecomposable subquandle $Q_x$ containing $x\in Q$.
\end{proposition}
\begin{proof}
It follows from the previous lemma that $\Gamma_n G(Q_x)=\Gamma_2 G(Q_x)$ for $n\geq 2$. Hence, the morphism $f_x$ extending the restriction $f_{\vert Q_x}$ of $f$ to $G(Q_x)$ maps $\Gamma_n G(Q_x)$ to zero and hence $f_x$ factors throught $G(Q_x)^{ab}$ where $Q_x$ is reduced to one point (corollary \ref{clas}). This proves the proposition.
\end{proof}
\begin{corollary}
The indecomposable subquandles of $Q$ a subquandle of $Conj(G)$ for $G$ nilpotent are single points.
\end{corollary}
\begin{definition}[\cite{dar}]
A quandle $Q$ is nilpotent if $Inn(Q)$ is a nilpotent group.
\end{definition}
The first two conditions in the following proposition are known to be equivalent (\cite{dar}), we give a slight different proof.
\begin{proposition}
For $Q$ a quandle the following three conditions are equivalent : $G(Q)$ is nilpotent, $Q$ is nilpotent, $Q_{conj}$ is nilpotent.
\end{proposition}
\begin{proof}
By corollary \ref{cen}, $G(Q)\simeq G(Q_{conj})$ is a central extension of $Inn(Q_{conj})$. Hence $G(Q)$ is nilpotent is nilpotent if and only if $Inn(Q_{conj})$ is. One has a quandle morphism from $Q$ to $Inn(Q)$ mapping $x\in Q$ to the left translation $L_x$. This morphism induces a surjective morphism $G(Q) \to Inn(Q)$ since $Inn(Q)$ is generated by the left translations. This proves that if $G(Q)$ is nilpotent then $Inn(Q)$ is nilpotent. Now since the natural quandle morphism $Q\to Q_{conj}$ is surjective we have a surjective group morphism $Inn(Q)\to Inn(Q_{conj})$ mapping a translation by $x\in Q$ to the translation by the image of $x$ in $Q_{conj}$. Indeed an since the generators $L_x$ for $x\in Q$ of $Inn(Q)$ stabilize the fibers of $Q\to Q_{conj}$ the elements of $Inn(Q)$ do also. This proves that if $Inn(Q)$ is nilpotent $Inn(Q_{conj})$ is nilpotent and hence $G(Q)$ is nilpotent. We have proved the proposition.
\end{proof}
\begin{proposition}
A quandle $Q$ is nilpotent if and only if $Q_{conj}$ is a subquandle of a nilpotent group (for conjuguacy). If $Q$ is injective, then $Q$ is nilpotent if and only if $Q$ is a subquandle of a nilpotent group.
\end{proposition}
\begin{proof}
If $Q$ is nilpotent the $Q_{conj}$ is a subquandle of a nilpotent group $G(Q)$ by the last proposition. We prove the converse. Let $Q_{conj}$ be a subquandle of a nilpotent group $G$. An let $H$ be the group of inner automorphisms of $G$ stabilizing $Q_{conj}$. $H$ is nilpotent since it is a subgroup of the inner automorphisms of $G$ wich is a nilpotent group. Restricting the elements of $H$ to $Q$ we get a morphism $f$ from $H$ to the group of bijections of $Q$. Notice that the image under $f$ of the conjugacy by $x\in Q$ is $L_x$, hence $Inn(Q)$ is included in the image of $H$. This proves that $Inn(Q)$ is nilpotent.  
\end{proof}
\begin{proposition}
An injective finite quandle is nilpotent if and only if it is a subquandle of a finite nilpotent group.
\end{proposition}
\begin{proof}
If the injective finite quandle is a subquandle of a finite nilpotent group then the quandle is nilpotent by the last proposition. We prove the converse using the proof of proposition \ref{fin} and by taking into account that $G(Q)$ is nilpotent if $Q$ is nilpotent and hence its quotients are nilpotent.
\end{proof}
\begin{remark}
If we define solvable quandles as quandles $Q$ where $Inn(Q)$ is solvable, the results of the last three propositions hold for solvable quandles if we replace the word nilpotent with solvable. Note also that for $Q$ finite, $G(Q)$ is polycyclic if and only if $Q$ is solvable. Indeed, as we have seen in the proof of theorem \ref{rep} the center of $G(Q)$ is a finitely generated abelian group and in the same section we prove that $G(Q)$ is an extension of $Inn(Q_{conj})$ by the center.
\end{remark}
\section{Algebraic invariant of the enveloping group} 
We recall that we denote by $gr G(Q)$ the associated graded of $G(Q)$ with respect to the lower central series $(\Gamma_n G(Q))_n$.
\begin{proposition}\label{gr}
\begin{itemize}
\item[1)] $gr_1 G(Q) \simeq  \Z (Q/Inn(Q))$.
\item[2)] For $Q$ finite and $i>1$, $gr_i G(Q)$ is a finite torsion group (eventually null). 
\end{itemize}
\end{proposition}
\begin{proof}
$1)$ follows from the definition of $gr G(Q)$ and proposition \ref{abel}. Now assume $Q$ is finite. For $x$ and $y$ in $Q$, denote by $x^n\triangleright y$ the $n$ times iteration $(x\triangleright(x\triangleright (\cdots (x\triangleright(x \triangleright y))\cdots)))$. Since $Q$ is finite $x^m\triangleright y=y$ for some $m$. Hence in $G(Q)$ $x^m y x^{-m}=y$ and hence we have that $(x^m,y)=1$ where$(-,-)$ is the commutator. This proves that : 
$$m[\bar{x},\bar{y}]=0,$$
where $\bar{x}$ and $\bar{y}$ are the classes of $x,y$ in the $G(Q)^{ab}$ and $[-,-]$ denotes the Lie bracket of $grG(Q)$. Therefore, $[\bar{x},\bar{y}]$ is a torsion element.  Since $Q$ genrates $G(Q)$ and $grG(Q)$ is generaterd as a Lie ring by $gr_1G(Q)=G(Q)^{ab}$ wich is finitely generated as a $\Z$ module, $gr_2 G(Q)$ is finite a torsion group. We then prove by induction using the fact that $gr G(Q)$ is generated as a Lie ring by its degree $1$ part that $gr_i G(Q)$ is a finite torsion group for all $i>2$.
\end{proof}
\begin{proposition}
The Malcev Lie algebra of $G(Q)$ for $Q$ finite is abelian and is isomorphic to $G(Q)^{ab}\otimes \mathbb{Q}\simeq \Z (Q/Inn(Q))\otimes \mathbb{Q}$.
\end{proposition}
\begin{proof}
The Malcev Lie algebra $L$ of $G(Q)$ is a filtered $\mathbb{Q}$-Lie algebra having a seperated filtration $(F_rL)_{r\geq 1}$ : $F_1L=L$, $F_rL \subset F_{r+1}L$, $(F_rL,F_sL) \subset F_{r+s}L$ and $\cap_{r\geq 1} F_r L=0$. $L$ is also complete with respect to the filtration (see \cite{RHT} for more details), but this will not be used. The associated graded $gr L$ of $L$ with respect to this filtration is isomorphic as a Lie algebra to the Lie algebra $grG(Q)\otimes \mathbb{Q}$. Since we assume that $Q$ is finite by the previous proposition we get that $gr_1 L\simeq  G(Q)^{ab}\otimes \mathbb{Q}$ and $gr_iL \simeq 0$ for $i\geq 2$. Hence $F_2L=F_3L=F_4L=\cdots$. But the intersection of the $F_rL$ is null. Therefore $F_2L=0$. This proves the proposition since $(L,L) \subset F_2L$ and $L/F_2L= gr_1L \simeq G(Q)^{ab}\otimes \mathbb{Q}$.
\end{proof}

\begin{proposition}
If $Q$ is a finite quandle then the rational cohomology ring $H^*(G(Q),\mathbb{Q})$ of the group $G(Q)$  is isomorphic to the exterior algebra $\Lambda^*\mathbb{Q}(Q/Inn(Q))$, where $$\mathbb{Q}(Q/Inn(Q))=\Z(Q/Inn(Q))\otimes \mathbb{Q}.$$
\end{proposition}
\begin{proof}
We can assume that $Q$ is injective since $G(Q_{conj})\simeq G(Q)$. We do so. Let $\mathcal{Z}$  be the center of $G(Q)$.  Since we assume $Q$ is finite $Inn(Q)$ is finite and $H^*(G(Q)/\mathcal{Z},\mathbb{Q})=H^0(G(Q)/\mathcal{Z},\mathbb{Q})=\mathbb{Q}$ ($G/\mathcal{Z}\simeq Inn(Q)$ by lemma \ref{lem1}). The center is of finite index and $G(Q)$ is finitely generated. Hence, $\mathcal{Z}$ is a finitely generated abelian group. In particular, we have a ring isomorphism $H^*(\mathcal{Z},\mathbb{Q})\simeq H^*(\Z, \mathbb{Q})^{\otimes r}$ for some $r$. Hence, the ring $H^*(\mathcal{Z},\mathbb{Q}) $ is isomorphic to $\Lambda^* \mathbb{Q}^r$ for some $r$. Now $G(Q)/\mathcal{Z} $ acts trivially in the rational cohomology of $\mathcal{Z}$. Hence, the second page of the Hochschil-serre spectral sequence for $\mathcal{Z}$ and $G(Q)$  over the rationals is given by :
$$E_2^{p,q}\simeq H^p(G(Q)/\mathcal{Z},\mathbb{Q})\otimes H^q(\mathcal{Z},\mathbb{Q}).$$
The spectral sequence collapses at the second page and $H^*(G(Q),\mathbb{Q})\simeq H^*(\mathcal{Z},\mathbb{Q})$ as graded vector spaces.  This with the fact that the restriction morphism $H^*(G(Q),\mathbb{Q})\to H^*(\mathcal{Z},\mathbb{Q})$ is an injective ring morphism and the fact that $ H^*(\mathcal{Z},\mathbb{Q})\simeq \Lambda^* \mathbb{Q}^r$ for some $r$, gives that $H^*(G(Q),\mathbb{Q})$ is isomorphic as ring to $\Lambda^* \mathbb{Q}^r$ for some $r$.  But, as we have seen $G(Q)^{ab}\simeq \Z(Q/Inn(Q))$. Hence $H^1(G(Q),\mathbb{Q})\simeq G(Q)^{ab}\otimes \mathbb{Q}$ and this determines $r$. We have proved the propostion
\end{proof}
We note that it follows from the proof that the free part of the center is isomorphic to $\Z(Q/Inn(Q))$. We give two remarks for $Q$ finite :
\begin{remark}
The cohomololgy of $G(Q)$ over the rationals correspond to the cohomology of its Malcev Lie algebra.
\end{remark}
\begin{remark}
The cohomology ring $H^*(G(Q),\mathbb{Q})$ is a Sulivan algebra hence equipped with the zero differential it is the minimal model of the classifying space $BG(Q)$ and $BG(Q)$ is formal in the sense of rational homotopy theory. 
\end{remark}
\begin{proposition}
If $Q$ is a finite nilpotent quandle, then $\Gamma_2G(Q)$ is a finite group it is the torsion subgroup of $G(Q)$.
\end{proposition}
\begin{proof}
Since we assume $Q$ is nilpotent $ \Gamma_mG(Q)=0$ some $m>4$. It follows, using proposition \ref{gr} (we assumed that $Q$ is finite) that $\Gamma_2 G(Q)$ is a finite group. Indeed, $\Gamma_2G(Q)/\Gamma_nG(Q)$ is an extension of $\Gamma_2 G(Q)/ \Gamma_{n-1} G(Q)$ by $gr_{n-1} G(Q)$. Hence $\Gamma_2G(Q)/\Gamma_m G(Q)=\Gamma_2 G(Q)$ is obtained as iterated extensions of the finite group $\Gamma_2 G(Q)/\Gamma_3 G(Q)$ by the finite groups $gr_i G(Q)$ for $3\leq i \leq m-1$. This proves that $\Gamma_2 G(Q)$ is a finite group. The fact that it contains all the torsion elements is due to the fact that it is the kernel of a morphism from $G(Q)$ to a free $\Z$-module : $G(Q)\to G(Q)/\Gamma_2 G(Q) \simeq \Z (Q/Inn(Q))$. 
\end{proof}
\begin{corollary}
For $Q$ finite and nilpotent the group $G(Q)$ is an extension of $\Z(Q/Inn(Q))$ by a finite group the group $\Gamma_2 G(Q)$.
\end{corollary}

 \end{document}